\documentclass[12pt]{amsart}                                                                                 %
\usepackage{color}                                                                                           %
\usepackage[colorlinks,linkcolor=blue,citecolor=blue]{hyperref}                                              %
\usepackage{graphicx}                                                                                        %
\usepackage[T1]{fontenc}                                                                                     %
\usepackage[latin1]{inputenc}                                                                                %
\usepackage{ae}                                                                                              %
\usepackage{amsmath}                                                                                         %
\usepackage{amsfonts}                                                                                        %
\usepackage{amssymb}                                                                                         %
\usepackage{amsxtra}                                                                                         %
\usepackage{dsfont}                                                                                          %
\usepackage{eucal}                                                                                           %
\usepackage{latexsym}                                                                                        %
\usepackage{amsthm}                                                                                          %
\usepackage{array}                                                                                           %
\usepackage{eqnarray}                                                                                        %
\usepackage{paralist}
\usepackage{cite}
\theoremstyle{plain}                                                                                         %
\newtheorem{thm}{Theorem}[section]                                                                         %
\newtheorem{prop}[thm]{Proposition}                                                                        %
\newtheorem{cor}[thm]{Corollary}                                                                           %
\theoremstyle{definition}                                                                                    %
\newtheorem{exam}[thm]{Example}                                                                         %
\theoremstyle{remark}                                                                                        %
\newtheorem{rmk}[thm]{Remark}
\pdfstringdef{\Title}{Operators on $\mathcal{C}(\Omega)\hat{\otimes}_{d_{p}}X$}                                 %
\hypersetup{pdftitle=\Title}                                                                                 %
\pdfstringdef{\Author}{Fernando Mu\~{n}oz, Eve Oja, and C\'{a}ndido Pi\~{n}eiro}                                              %
\hypersetup{pdfauthor=\Author}                                                                               %
\hypersetup{pdfstartview=FitH}                                                                               %

\begin{document}

\baselineskip=17pt

\title[Operators on $\mathcal{C}_{p}(\Omega,X)$]{Operators on the Banach Space of $p$-continuous vector-valued functions}

\author{Fernando Mu\~{n}oz, Eve Oja, and C\'{a}ndido Pi\~{n}eiro}



\address{Departamento de Matem\'aticas, Facultad de Ciencias Experimentales, Universidad de Huelva, Campus Universitario de El Carmen, 21071 Huelva, Spain}

\email{fmjimenez@dmat.uhu.es}

\address{Faculty of Mathematics and Computer Science, University of Tartu, J. Liivi 2, 50409 Tartu, Estonia; Estonian Academy of Sciences, Kohtu 6, 10130 Tallinn, Estonia}

\email{eve.oja@ut.ee}

\address{Departamento de Matem\'aticas, Facultad de Ciencias Experimentales, Universidad de Huelva, Campus Universitario de El Carmen, 21071 Huelva, Spain}

\email{candido@uhu.es}

\maketitle
\begin{center}
\textit{{\small Dedicated to Professor Nicolae Dinculeanu on his ninetieth birthday}}
\end{center}

\begin{abstract}
  Let $X$, $Y$, and $Z$ be Banach spaces, and let $\alpha$ be a tensor norm. Let a bounded linear operator $S\in\mathcal{L}(Z,\mathcal{L}(X,Y))$ be given. We obtain (necessary and/or sufficient) conditions for the existence of an operator $U\in\mathcal{L}(Z\hat{\otimes}_{\alpha}X,Y)$ such that $(Sz)x = U(z\otimes x)$, for all $z\in Z$ and $x\in X$, i.e., $S= U^{\#}$, the associated operator to $U$. Let $\Omega$ be a compact Hausdorff space and denote by $\mathcal{C}(\Omega)$ the space of continuous functions from $\Omega$ into $\mathds{K}$. We apply these results to $S\in\mathcal{L}(\mathcal{C}(\Omega),\mathcal{L}(X, Y))$ for characterizing the existence of an operator $U\in\mathcal{L}(\mathcal{C}_{p}(\Omega,X),Y)$ such that $U^{\#}=S$, where $\mathcal{C}_{p}(\Omega,X)$ is the space of $p$-continuous $X$-valued functions, $1\leq p \leq \infty$.
\end{abstract}

\section{Introduction}\label{s1}

Let $X$ and $Y$ be Banach spaces and let $\Omega$ be a compact Hausdorff space. The space of continuous functions from $\Omega$ into $X$ ($\mathds{K}$, respectively) is denoted by $\mathcal{C}(\Omega,X)$ ($\mathcal{C}(\Omega)$, respectively). Let $\mathcal{L}(X,Y)$ denote the Banach space of bounded linear operators from $X$ into $Y$. For every operator $U\in\mathcal{L}(\mathcal{C}(\Omega,X), Y)$, we denote by $U^{\#}$ the \emph{associated operator} from $\mathcal{C}(\Omega)$ to $\mathcal{L}(X,Y)$ defined by $(U^{\#}\varphi)x = U(\varphi x)$, $\varphi\in\mathcal{C}(\Omega)$ and $x\in X$. (The notation $U^{\#}$ is traditional; see, e.g., \cite{M-SS, Po1, Po2, Po3, Po4, SS, Sw}.) Then, clearly, $U^{\#}\in\mathcal{L}(\mathcal{C}(\Omega),\mathcal{L}(X,Y))$.

On the other hand, in a short remark (see \cite[Remark, p. 379]{D}), Dinculeanu pointed out that there exist operators $S\in\mathcal{L}(\mathcal{C}(\Omega),\mathcal{L}(X, Y))$ which are not associated to any $U\in\mathcal{L}(\mathcal{C}(\Omega, X),Y)$, meaning that $S \neq U^{\#}$ for all operators $U\in\mathcal{L}(\mathcal{C}(\Omega, X),Y)$. (In \cite{D}, $U^{\#}$ is denoted by $U'$.) Professor Dinculeanu kindly informed us (personal communication, September 27, 2015) that his remark was just a conjecture based on Grothendieck's result quoted in Remark \ref{r:2.4} below.

Many authors have studied the interplay between $U$ and $U^{\#}$ for different classes of operators (see, e.g., the above references). However, it seems that nothing (apart from Dinculeanu's remark) has been said about the problem of the existence of an operator $U$ such that $U^{\#}=S$ for a given operator $S$.

This paper aims in studying this existence problem and, in particular, in proving Dinculeanu's conjecture. However, we shall study the problem in a more general context of operators defined on the Banach space $\mathcal{C}_{p}(\Omega,X)$ of $p$-continuous $X$-valued functions (see Section \ref{s3} for the definition and references). Since $\mathcal{C}_{\infty}(\Omega,X) = \mathcal{C}(\Omega,X)$, this also encompasses the classical case of operators on $\mathcal{C}(\Omega,X)$.

By Grothendieck's classics \cite{G2} (see, e.g., \cite[pp. 49--50]{R}), we know that
\[
\mathcal{C}(\Omega,X) = \mathcal{C}(\Omega)\hat{\otimes}_{\varepsilon}X,
\]
where $\varepsilon$ denotes the injective tensor norm, under the canonical isometric isomorphism $\varphi x \leftrightarrow \varphi\otimes x$, $\varphi\in\mathcal{C}(\Omega)$ and $x\in X$. As is well known, this allows to extend the definition of $U^{\#}$ as follows.

Let $Z$ be a Banach space and let $\alpha$ be a tensor norm. If $U\in\mathcal{L}(Z\hat{\otimes}_{\alpha}X,Y)$, then the operator $U^{\#}\in\mathcal{L}(Z,\mathcal{L}(X,Y))$ \emph{associated to} $U$ is defined by $(U^{\#}z)x = U(z\otimes x)$, $z\in Z$ and $x\in X$. By a recent result of the authors \cite[Theorem 3.8]{MOP1},
\[
\mathcal{C}_{p}(\Omega,X)=\mathcal{C}(\Omega)\hat{\otimes}_{d_{p}}X,
\]
where $d_{p}$ denotes the right Chevet--Saphar tensor norm (see \cite{S} or, e.g., \cite[Chapter 6]{R}). Keeping this in mind, we shall study the existence problem in the general context of operators defined on tensor products of Banach spaces. In particular, we shall see that examples, proving Dinculeanu's conjecture, come out on the all three levels of the generality (see Remarks \ref{r:2.2} and \ref{r:2.4}, Corollary \ref{c:4equiv}, Proposition \ref{p:condnecesuf}).

Let $Z$, $X$, and $\alpha$ be as above. In Section \ref{s3}, we prove a general omnibus theorem (Theorem \ref{t:4equiv-2-Eve}), which provides three equivalent conditions for the existence of $U\in\mathcal{L}(Z\hat{\otimes}_{\alpha}X,Y)$ such that $S=U^{\#}$ for every Banach space $Y$ and every operator $S\in\mathcal{L}(Z,\mathcal{L}(X,Y))$. The main applications (Theorem \ref{c:3equiv-2} and Corollaries \ref{p:cotype2} and \ref{p:cotypeq}) concern the case of $p$-continuous  $X$-valued functions $\mathcal{C}_{p}(\Omega,X) = \mathcal{C}(\Omega)\hat{\otimes}_{d_{p}}X$.

In Section \ref{s2}, we fix Banach spaces $Z$, $X$, and $Y$, and a tensor norm $d_{p}$. We prove another omnibus theorem (Theorem \ref{t:3equiv}), which provides four equivalent conditions for a given operator $S\in\mathcal{L}(Z,\mathcal{L}(X,Y))$ to be the associated operator to an operator $U\in\mathcal{L}(Z\hat{\otimes}_{d_{p}}X,Y)$. Again, the main application (Corollary \ref{c:4equiv}) concerns $\mathcal{C}_{p}(\Omega,X)$, yielding conditions that seem to be new even in the classical case $\mathcal{C}(\Omega,X) = \mathcal{C}_{\infty}(\Omega,X)$.

In Section \ref{s5}, we are given an operator $S\in\mathcal{L}(\mathcal{C}(\Omega),\mathcal{L}(X,Y))$. We present a necessary condition for the existence of $U\in\mathcal{L}(\mathcal{C}_{p}(\Omega,X),Y)$ such that $S=U^{\#}$ (Proposition \ref{p:condnecess}), which becomes also sufficient in the case $\mathcal{C}(\Omega,X) = \mathcal{C}_{\infty}(\Omega,X)$ (Proposition \ref{p:condnecesuf}). This condition is expressed in terms of the representing measure of $S$, so we build the representing measure of such kind of operators.

Section \ref{s4} provides three examples (concerning Corollaries \ref{p:cotypeq}, \ref{Uexist}, and Proposition \ref{p:condnecess}) in order to show that our results are sharp in general.

Our notation is standard. Let $1\leq p \leq \infty$, and denote by $p'$ the conjugate index of $p$ (i.e., $1/p + 1/p' = 1$ with the convention $1/\infty = 0$). We consider Banach spaces over the same, either real or complex, field $\mathds{K}$. The closed unit ball of $X$ is denoted by $B_{X}$. The Banach space of all \emph{absolutely $p$-summable sequences} in $X$ is denoted by $\ell_{p}(X)$ and its norm by $\|\cdot\|_{p}$. By $\ell_{p}^{w}(X)$ we mean the Banach space of \emph{weakly $p$-summable sequences} in $X$ with the norm $\|\cdot\|_{p}^{w}$ (see, e.g., \cite[pp. 32--33]{DJT}). Denote by $\ell_{p}^{u}(X)$ the Banach space of all \emph{unconditionally $p$-summable sequences} in $X$, which is the closed subspace of $\ell_{p}^{w}(X)$ formed by the sequences $(x_{n})\in\ell_{p}^{w}(X)$ satisfying $(x_{n}) = \lim_{N\rightarrow\infty}(x_{1},\ldots,x_{N},0,0,\ldots)$ in $\ell_{p}^{w}(X)$ (see \cite{FS1} or, e.g., \cite[8.2, 8.3]{DF}). We refer to Pietsch's book \cite{P} for the theory of operator ideals and, in particular, to the book \cite{DJT} by Diestel, Jarchow, and Tonge for absolutely $(r,q)$- and $q$-summing operators. Our main reference on the theory of tensor norms and related Banach operator ideals is the book of Ryan \cite{R}.

\section{Characterizing associated operators: ``global'' case}\label{s3}

Let $X$, $Y$, and $Z$ be Banach spaces, and let $\pi$ be the projective tensor norm. It is well known that every operator $U\in\mathcal{L}(Z\otimes_{\pi}X,Y)$ induces an \emph{associated operator} $U^{\#}\in\mathcal{L}(Z,\mathcal{L}(X,Y))$ by
\[
(U^{\#}z)x = U(z\otimes x), \, z\in Z \mbox{ and } x\in X.
\]
It is also well known and easy to verify that the correspondence $U\mapsto U^{\#}$ is an isometric isomorphism between the Banach spaces $\mathcal{L}(Z\otimes_{\pi}X,Y) = \mathcal{L}(Z\hat{\otimes}_{\pi}X,Y)$ and $\mathcal{L}(Z,\mathcal{L}(X,Y))$. In particular, every $S\in\mathcal{L}(Z,\mathcal{L}(X,Y))$ happens to be the associated operator to some $U\in\mathcal{L}(Z\hat{\otimes}_{\pi}X,Y)$.

In the special case when $Y=\mathds{K}$, the operators $U$ and $U^{\#}$ are canonically identified, and the corresponding identification
\[
(Z\hat{\otimes}_{\pi}X)^{*} = \mathcal{L}(Z,X^{*})
\]
(as Banach spaces) uses the duality
\[
\langle S, \sum_{i=1}^{n}z_{i}\otimes x_{i} \rangle = \sum_{i=1}^{n}(Sz_{i})x_{i},
\]
i.e., $S=U^{\#}$ is identified with $U$.

The same phenomenon occurs for any tensor norm $\alpha$: thanks to Grothendieck \cite{G} (see, e.g., \cite[pp. 187--190]{R}), one has the canonical identification
\[
(Z\hat{\otimes}_{\alpha}X)^{*} = \mathcal{A}(Z, X^{*})
\]
(as Banach spaces), where $\mathcal{A}$ is the Banach operator ideal of the $\alpha'$-integral operators. Following \cite{OR}, let us say that $\mathcal{A}$ is the \emph{dual space operator ideal} of $\alpha$. (Note that in \cite{O2}, the dual space operator ideal was defined differently, but in a symmetric way.)

Let $\alpha$ be a tensor norm. Since $\alpha(u)\leq\pi(u)$, $u\in Z\otimes X$, we have $\mathcal{L}(Z\hat{\otimes}_{\alpha}X, Y)\subset\mathcal{L}(Z\hat{\otimes}_{\pi}X, Y)$. Therefore the associated operator $U^{\#}\in\mathcal{L}(Z,\mathcal{L}(X,Y))$ is defined for any $U\in\mathcal{L}(Z\hat{\otimes}_{\alpha}X, Y)$, and, in particular, for any $U\in\mathcal{A}(Z\hat{\otimes}_{\alpha}X, Y)$, where $\mathcal{A}=(\mathcal{A}, \|\cdot\|_{\mathcal{A}})$ is an \emph{arbitrary} Banach operator ideal. However, in this case, in general, not all $S\in\mathcal{L}(Z,\mathcal{L}(X,Y))$ enjoy the ``privilege'' of being associated to some $U\in\mathcal{L}(Z\hat{\otimes}_{\alpha}X, Y)$. Theorem \ref{t:4equiv-2-Eve} below will give three equivalent conditions for the existence of such $U\in\mathcal{L}(Z\hat{\otimes}_{\alpha}X,Y)$ for every Banach space $Y$ and every operator $S\in\mathcal{L}(Z,\mathcal{L}(X,Y))$.

\begin{thm} \label{t:4equiv-2-Eve}
Let $X$ and $Z$ be Banach spaces. Let $\alpha$ be a tensor norm and let $\mathcal{A}$ be the dual space operator ideal of $\alpha$. The following statements are equivalent.

$($\emph{a}$)$ For every Banach space $Y$ and for every operator $S\in\mathcal{L}(Z,\mathcal{L}(X,Y))$, there exists $U\in\mathcal{L}(Z\hat{\otimes}_{\alpha}X,Y)$ such that $U^{\#} = S$.

$($\emph{b}$)$ There exists a Banach space $Y\neq \{0\}$ such that for every operator $S\in\mathcal{L}(Z,\mathcal{L}(X,Y))$, there exists $U\in\mathcal{L}(Z\hat{\otimes}_{\alpha}X,Y)$ such that $U^{\#} = S$.

$($\emph{c}$)$ $\mathcal{L}(Z,X^{*}) = \mathcal{A}(Z,X^{*})$ as sets.

$($\emph{d}$)$ The tensor norms $\alpha$ and $\pi$ are equivalent on $Z\otimes X$.
\end{thm}

\begin{proof}
(a)$\Rightarrow$(b). This is trivial.

(b)$\Rightarrow$(c). Fix $y_{0}\in Y$ and $y_{0}^{*}\in Y^{*}$ satisfying $\|y_{0}\| = \|y_{0}^{*}\| = y_{0}^{*}(y_{0}) = 1$. Let $A\in\mathcal{L}(Z,X^{*})$. Define $S:Z\rightarrow \mathcal{L}(X,Y)$ by $Sz = Az\otimes y_{0}\in\mathcal{F}(X,Y)$. Then $S\in\mathcal{L}(Z,\mathcal{L}(X,Y))$ and there exists $U\in\mathcal{L}(Z\hat{\otimes}_{\alpha}X,Y)$ such that $U^{\#}=S$. Let us consider
\[
y_{0}^{*}U\in\mathcal{L}(Z\hat{\otimes}_{\alpha}X,\mathds{K}) = (Z\hat{\otimes}_{\alpha}X)^{*} = \mathcal{A}(Z,X^{*}).
\]
Since
\[
(y_{0}^{*}U)(z\otimes x) = y_{0}^{*}(U(z\otimes x)) = y_{0}^{*}((Sz)x)
\]
\[
= y_{0}^{*}((Az)(x)y_{0}) = (Az)(x),\,\,x\in X, z\in Z,
\]
$y_{0}^{*}U = A$; hence $A\in\mathcal{A}(Z,X^{*})$.

(c)$\Rightarrow$(d). We know that
\[
(Z\otimes_{\pi}X)^{*} = (\mathcal{L}(Z,X^{*}),\|\cdot\|)\,\,\mbox{ and }\,\,(Z\otimes_{\alpha}X)^{*} = (\mathcal{A}(Z,X^{*}),\|\cdot\|_{\mathcal{A}})
\]
as Banach spaces. By (c), the Banach space $(\mathcal{L}(Z,X^{*}),\|\cdot\|)$ also carries another complete norm $\|\cdot\|_{\mathcal{A}}$. Since, as is well known, $\|\cdot\|\leq\|\cdot\|_{\mathcal{A}}$, the norms $\|\cdot\|$ and $\|\cdot\|_{\mathcal{A}}$ are equivalent. Hence also $\pi$ and $\alpha$ are equivalent.

(d)$\Rightarrow$(a). Let $Y$ be a Banach space. By the canonical isometric isomorphism between $\mathcal{L}(Z,\mathcal{L}(X,Y))$ and $\mathcal{L}(Z\hat{\otimes}_{\pi}X,Y)$, for every $S\in\mathcal{L}(Z,\mathcal{L}(X,Y))$, there exists $U\in\mathcal{L}(Z\hat{\otimes}_{\pi}X,Y)$ such that $U^{\#}=S$. But $\mathcal{L}(Z\hat{\otimes}_{\pi}X,Y) = \mathcal{L}(Z\hat{\otimes}_{\alpha}X,Y)$ as sets (and isomorphic as Banach spaces), because $\pi$ and $\alpha$ are equivalent on $Z\otimes X$.
\end{proof}

\begin{rmk} \label{r:2.2}
The particular case of $\alpha=\varepsilon$ concerns one of the most famous long-standing conjectures in functional analysis. In \cite[p. 153]{G2} (see also \cite[Section 4.6]{G}), Grothendieck conjectured: if the injective tensor norm $\varepsilon$ and the projective tensor norm $\pi$ are equivalent on $Z\otimes X$, then $Z$ or $X$ must be finite dimensional. In 1981, Pisier \cite{Pis} constructed an infinite-dimensional separable Banach space $P$ such that $\varepsilon$ and $\pi$ are equivalent on  $P\otimes P$. Since $\varepsilon(u)\leq\alpha(u)\leq\pi(u)$, $u\in Z\otimes X$, all tensor norms are equivalent on $P\otimes P$. By Theorem \ref{t:4equiv-2-Eve}, for every tensor norm $\alpha$ and every Banach space $Y$, every operator $S\in\mathcal{L}(P,\mathcal{L}(P,Y))$ is associated to some $U\in\mathcal{L}(P\hat{\otimes}_{\alpha}P,Y)$.

In \cite[p. 153, Corollary 2]{G2} (see \cite[Corollary 3.1]{O2} for a different proof), it is proved that if $\varepsilon$ and $\pi$ are equivalent on $Z\otimes Z^{*}$, then $Z$ is finite dimensional. Again, by Theorem \ref{t:4equiv-2-Eve}, for every infinite-dimensional Banach space $Z$ and every Banach space $Y\neq \{0\}$, there exists an operator $S\in\mathcal{L}(Z,\mathcal{L}(Z^{*},Y))$ which is not associated to any operator $U\in\mathcal{L}(Z\hat{\otimes}_{\varepsilon} Z^{*}, Y)$. With $Z=\mathcal{C}(\Omega)$, this clearly can be used to prove Dinculeanu's conjecture (see the Introduction).
\end{rmk}

Let $1\leq p \leq \infty$. In \cite{MOP1}, the authors considered the Banach space of \emph{$p$-continuous} $X$-valued functions $\mathcal{C}_{p}(\Omega, X)$ formed by all $f\in\mathcal{C}(\Omega,X)$ such that $f(\Omega)$ is $p$-compact (i.e., there exists a sequence $(x_{n})\in\ell_{p}(X)$ (or $(x_{n})\in c_{0}(X)$ when $p=\infty$) such that $f(\Omega)\subset \{ \sum_{n}\alpha_{n}x_{n} \, : \, (\alpha_{n})\in B_{\ell_{p'}} \}$). It follows from properties of $p$-compactness (see, e.g., \cite{SK}) that $\mathcal{C}_{p}(\Omega, X) \subset \mathcal{C}_{q}(\Omega, X)$ if $p\leq q$, and $\mathcal{C}_{\infty}(\Omega, X) = \mathcal{C}(\Omega, X)$. The space $\mathcal{C}_{p}(\Omega, X)$ becomes a Banach space endowed with the norm
\[
\|f\|_{\mathcal{C}_{p}(\Omega, X)}=\inf \|(x_{n})\|_{p},
\]
where the infimum is taken over all sequences $(x_{n})\in\ell_{p}(X)$ (or $(x_{n})\in c_{0}(X)$ when $p=\infty$) such that $f(\Omega)\subset \{ \sum_{n}\alpha_{n}x_{n} \, : \, (\alpha_{n})\in B_{\ell_{p'}} \}$, and $\mathcal{C}_{\infty}(\Omega, X) = \mathcal{C}(\Omega, X)$ as Banach spaces (see \cite[Proposition 3.6]{MOP1}).

One of the main results of \cite{MOP1} is that (as was mentioned in the Introduction) $\mathcal{C}_{p}(\Omega, X) = \mathcal{C}(\Omega)\hat{\otimes}_{d_{p}}X$ as Banach spaces. It is known (see, e.g., \cite[p. 142]{R}) that the dual space operator ideal of the Chevet--Saphar tensor norm $d_{p}$ coincides with $\mathcal{P}_{p'}$, where $\mathcal{P}_{q}$, $1\leq q\leq \infty$, denotes the Banach operator ideal of absolutely $q$-summing operators. This leads us to the following immediate application of Theorem \ref{t:4equiv-2-Eve}.

\begin{thm} \label{c:3equiv-2}
Let $X$ be a Banach space and let $\Omega$ be a compact Hausdorff space. Let $1\leq p\leq\infty$. The following statements are equivalent.

$($\emph{a}$)$ For every Banach space $Y$ and for every operator $S\in\mathcal{L}(\mathcal{C}(\Omega), \mathcal{L}(X,Y))$, there exists $U\in\mathcal{L}(\mathcal{C}_{p}(\Omega,X), Y)$ such that $U^{\#} = S$.

$($\emph{b}$)$ There exists a Banach space $Y\neq \{0\}$ such that for every operator $S\in\mathcal{L}(\mathcal{C}(\Omega),\mathcal{L}(X,Y))$, there exists $U\in\mathcal{L}(\mathcal{C}_{p}(\Omega,X),Y)$ such that $U^{\#} = S$.

$($\emph{c}$)$ $\mathcal{L}(\mathcal{C}(\Omega),X^{*}) = \mathcal{P}_{p'}(\mathcal{C}(\Omega),X^{*})$ as sets.

$($\emph{d}$)$ The tensor norms $d_{p}$ and $\pi$ are equivalent on $\mathcal{C}(\Omega)\otimes X$.
\end{thm}

\begin{rmk} \label{r:2.4}
As was mentioned, $\mathcal{C}(\Omega,X) = \mathcal{C}_{\infty}(\Omega,X)$. Saphar \cite[p. 99]{S} has shown
that $d_{\infty}$ coincides with $\varepsilon$ on $\mathcal{C}(\Omega)\otimes X$. But, thanks to Grothendieck \cite[p. 152, Proposition 33]{G2}, $\varepsilon$ and $\pi$ cannot be equivalent on $\mathcal{C}(\Omega)\otimes X$ when $X$ is
an infinite-dimensional Banach space (assuming, of course, that $\Omega$ is infinite). Thus, in this special case, Theorem \ref{c:3equiv-2} says that, whenever $Y \neq \{0\}$, there exists an operator $S\in\mathcal{L}(\mathcal{C}(\Omega), \mathcal{L}(X,Y))$ which is not associated to any $U\in\mathcal{L}(\mathcal{C}(\Omega,X),Y)$. This refines Dinculeanu's remark and proves his conjecture (see the Introduction).
\end{rmk}

Since $\mathcal{P}_{\infty} = \mathcal{L}$, the following is immediate from Theorem \ref{c:3equiv-2}.

\begin{cor} \label{c:C1}
Let $X$ be a Banach space and let $\Omega$ be a compact Hausdorff space. For every Banach space $Y$ and every operator $S\in\mathcal{L}(\mathcal{C}(\Omega), \mathcal{L}(X,Y))$, there exists $U\in\mathcal{L}(\mathcal{C}_{1}(\Omega,X),Y)$ such that $U^{\#} = S$.
\end{cor}

Equality (c) of Theorem \ref{c:3equiv-2}, which is well studied in the literature (see, e.g., \cite[Chapter 11]{DJT} for results and references), enables us to move from $\mathcal{C}_{1}(\Omega,X)$ to larger domain spaces $\mathcal{C}_{p}(\Omega,X)$ for special cases of $X$.

\begin{cor} \label{p:cotype2}
Let $X$ and $Y$ be Banach spaces such that $X^{*}$ is of cotype 2. Let $\Omega$ be a compact Hausdorff space. Assume that $S\in\mathcal{L}(\mathcal{C}(\Omega),\mathcal{L}(X,Y))$. Then, for each $p\leq2$, there exists an operator $U\in\mathcal{L}(\mathcal{C}_{p}(\Omega,X), Y)$ such that $U^{\#} = S$.
\end{cor}

\begin{proof}
Let $X^{*}$ be of cotype 2. Then $\mathcal{L}(\mathcal{C}(\Omega), X^{*}) = \mathcal{P}_{2}(\mathcal{C}(\Omega), X^{*})$ (see, e.g., \cite[Theorem 11.14]{DJT}). Since $p'\geq2$, we have $\mathcal{P}_{2}(\mathcal{C}(\Omega), X^{*})\subset \mathcal{P}_{p'}(\mathcal{C}(\Omega), X^{*})$ (see, e.g., \cite[Theorem 2.8]{DJT}). Hence, $\mathcal{L}(\mathcal{C}(\Omega), X^{*}) = \mathcal{P}_{p'}(\mathcal{C}(\Omega), X^{*})$ and we only need to apply Theorem \ref{c:3equiv-2} to finish the proof.
\end{proof}

In Section \ref{s4}, we shall show that Corollary \ref{p:cotype2} does not hold for $p>2$ (see Example \ref{ex:cotype2NO}).

\begin{cor} \label{p:cotypeq}
Let $X$ and $Y$ be Banach spaces such that $X^{*}$ is of cotype $q$, where $2\leq q <\infty$. Let $\Omega$ be a compact Hausdorff space. Assume that $S\in\mathcal{L}(\mathcal{C}(\Omega),\mathcal{L}(X,Y))$. Then, for each $p<q'$, there exists an operator $U\in\mathcal{L}(\mathcal{C}_{p}(\Omega,X),Y)$ such that $U^{\#} = S$.
\end{cor}

\begin{proof}
Let $X^{*}$ be of cotype $q$, $2\leq q<\infty$. Then $\mathcal{L}(\mathcal{C}(\Omega), X^{*}) = \mathcal{P}_{r}(\mathcal{C}(\Omega), X^{*})$, for all $r>q$ (see, e.g., \cite[Theorem 11.14]{DJT}). Since $p'>q$, $\mathcal{L}(\mathcal{C}(\Omega), X^{*}) = \mathcal{P}_{p'}(\mathcal{C}(\Omega), X^{*})$ and the proof finishes using again Theorem \ref{c:3equiv-2}.
\end{proof}

\section{Characterizing associated operators: ``local'' case}\label{s2}

In this section, let an operator $S\in\mathcal{L}(\mathcal{C}(\Omega),\mathcal{L}(X,Y))$ be given, where $X$ and $Y$ are Banach spaces. Let $1\leq p\leq\infty$. We are interested in equivalent conditions for the existence of an operator $U\in\mathcal{L}(\mathcal{C}_{p}(\Omega,X),Y)$ such that $U^{\#} = S$. These conditions will be presented in Corollary \ref{c:4equiv} below. They immediately follow from the more general case when $S\in\mathcal{L}(Z,\mathcal{L}(X,Y))$ and $U\in\mathcal{L}(Z\hat{\otimes}_{d_{p}}X,Y)$, where $Z$ is a Banach space, see Theorem \ref{t:3equiv}. Theorem \ref{t:3equiv} in turn will be deduced from Theorem \ref{t:3Eve}, which characterizes operators that take $Z$ to the space $\mathcal{P}_{(r,q)}(X,Y)$ of absolutely $(r,q)$-summing operators, and is the main result of this section.

Let $1\leq q\leq r\leq\infty$. Recall that an operator $U\in\mathcal{L}(X,Y)$ is \emph{absolutely $(r,q)$-summing} if there is a constant $C\geq 0$ such that
\[
\|(Ux_{i})_{i=1}^{n}\|_{r} \leq C\, \|(x_{i})_{i=1}^{n}\|_{q}^{w}
\]
for all finite systems $(x_{i})_{i=1}^{n}\subset X$. The least constant $C$ for which the previous equality holds is denoted by $\|U\|_{\mathcal{P}_{(r,q)}}$. All absolutely $(r,q)$-summing operators between arbitrary Banach spaces form a Banach operator ideal, denoted by $\mathcal{P}_{(r,q)}$. Recall also that the Banach operator ideal $\mathcal{P}_{q}$ of absolutely $q$-summing operators is defined as $\mathcal{P}_{q} = \mathcal{P}_{(q,q)}$.

\begin{thm} \label{t:3Eve}
Let $X$, $Y$, and $Z$ be Banach spaces. Let $1\leq q\leq r\leq \infty$. Assume that $T\in\mathcal{L}(Z,\mathcal{L}(X,Y))$. The following statements are equivalent.

\emph{(a)} $T\in\mathcal{L}(Z,\mathcal{P}_{(r,q)}(X,Y))$.

\emph{(b)} There exists a constant $c>0$ such that, for all finite systems $(y_{i}^{*})_{i=1}^{n}\subset Y^{*}$ and $(x_{i})_{i=1}^{n}\subset X$,
\[
\|(T^{*}(x_{i}\otimes y_{i}^{*}))\|_{r}^{w} \leq c\,\|(y_{i}^{*})\|_{\infty}\|(x_{i})\|_{q}^{w}.
\]

\emph{(b$'$)} There exists a constant $c>0$ such that, for all $(y_{i}^{*})\in \ell_{\infty}(Y^{*})$ and $(x_{i})\in \ell_{q}^{w}(X)$, and for all $n\in\mathds{N}$,
\[
\|(T^{*}(x_{i}\otimes y_{i}^{*}))_{i=n}^{\infty}\|_{r}^{w} \leq c\,\|(y_{i}^{*})_{i=n}^{\infty}\|_{\infty}\|(x_{i})_{i=n}^{\infty}\|_{q}^{w}.
\]

\emph{(c)} If $(y_{i}^{*})\in \ell_{\infty}(Y^{*})$ and $(x_{i})\in \ell_{q}^{w}(X)$, then $(T^{*}(x_{i}\otimes y_{i}^{*}))\in\ell_{r}^{w}(Z^{*})$.

\emph{(d)} If $(y_{i}^{*})\in c_{0}(Y^{*})$ and $(x_{i})\in \ell_{q}^{w}(X)$ $($or $(y_{i}^{*})\in \ell_{\infty}(Y^{*})$ and $(x_{i})\in \ell_{q}^{u}(X)$$)$, then $(T^{*}(x_{i}\otimes y_{i}^{*}))\in\ell_{r}^{u}(Z^{*})$.
\end{thm}

\begin{proof}
(a)$\Leftrightarrow$(b). Condition (a) is equivalent to the existence of a constant $c>0$ such that
\[
\|Tz\|_{\mathcal{P}_{(r,q)}} \leq c\,\|z\|
\]
for all $z\in Z$. This means that, for all $z\in B_{Z}$ and finite systems $(x_{i})_{i=1}^{n}\subset X$,
\[
\|((Tz)x_{i})\|_{r} \leq c\,\|(x_{i})\|_{q}^{w}.
\]
Consider $y_{i}^{*}\in B_{Y^{*}}$ such that
\[
\|(Tz)x_{i}\|_{Y} = \sup_{y_{i}^{*}\in B_{Y^{*}}} | y_{i}^{*}((Tz)x_{i}) |
\]
for $i=1,\ldots,n$. We may write
\[
\|((Tz)x_{i})\|_{r}^{r} = \sum_{i=1}^{n} \sup_{y_{i}^{*}\in B_{Y^{*}}} | y_{i}^{*}((Tz)x_{i}) |^{r} = \sup_{\|(y_{i}^{*})_{i=1}^{n}\|_{\infty}\leq1} \sum_{i=1}^{n} | y_{i}^{*}((Tz)x_{i}) |^{r}.
\]
Hence, (a) is equivalent to the existence of $c>0$ such that, for all $(x_{i})_{i=1}^{n}\subset X$,
\[
\sup_{\|(y_{i}^{*})_{i=1}^{n}\|_{\infty}\leq1} \sup_{z\in B_{Z}} \Big( \sum_{i=1}^{n} | y_{i}^{*}((Tz)x_{i}) |^{r}   \Big)^{1/r} \leq c\,\|(x_{i})\|_{q}^{w}.
\]
Since
\[
\sup_{z\in B_{Z}} \Big( \sum_{i=1}^{n} | y_{i}^{*}((Tz)x_{i}) |^{r}   \Big)^{1/r} = \|(T^{*}(y_{i}^{*}\otimes x_{i}))\|_{r}^{w},
\]
the above inequality reads as
\[
\sup_{\|(y_{i}^{*})_{i=1}^{n}\|_{\infty}\leq1} \|(T^{*}(y_{i}^{*}\otimes x_{i}))\|_{r}^{w} \leq c\,\|(x_{i})\|_{q}^{w}.
\]
Therefore, (a) is clearly equivalent to (b).

For the equivalences of the remaining conditions, see Proposition \ref{p:t3Eve} below, where it is shown that conditions (b), (b$'$), (c), and (d) are equivalent in a more general context with $T^{*}$ replaced by an arbitrary continuous bilinear map.
\end{proof}

Recall that a separately continuous bilinear map is continuous (see, e.g., \cite[Theorem 1.2, p. 8]{DF}) and this is equivalent to be bounded (see, e.g., \cite[Proposition 1.1, p. 8]{DF}).

\begin{prop} \label{p:t3Eve}
Let $X$, $Y$, and $Z$ be Banach spaces. Let $1\leq q\leq r\leq \infty$. Assume that $A:X\times Y\rightarrow Z$ is a continuous bilinear map. The following statements are equivalent.

\emph{(i)} There exists a constant $c>0$ such that, for all finite systems $(x_{i})_{i=1}^{n}\subset X$ and $(y_{i})_{i=1}^{n}\subset Y$,
\[
\|(A(x_{i}, y_{i}))\|_{r}^{w} \leq c\,\|(y_{i})\|_{\infty}\|(x_{i})\|_{q}^{w}.
\]

\emph{(i$'$)} There exists a constant $c>0$ such that, for all $(x_{i})\in \ell_{q}^{w}(X)$ and $(y_{i})\in \ell_{\infty}(Y)$, and for all $n\in\mathds{N}$,
\[
\|(A(x_{i}, y_{i}))_{i=n}^{\infty}\|_{r}^{w} \leq c\,\|(y_{i})_{i=n}^{\infty}\|_{\infty}\|(x_{i})_{i=n}^{\infty}\|_{q}^{w}.
\]

\emph{(ii)} If $(x_{i})\in \ell_{q}^{w}(X)$ and $(y_{i})\in \ell_{\infty}(Y)$, then $(A(x_{i}, y_{i}))\in\ell_{r}^{w}(Z)$.

\emph{(iii)} If $(x_{i})\in \ell_{q}^{w}(X)$ and $(y_{i})\in c_{0}(Y)$ $($or $(x_{i})\in \ell_{q}^{u}(X)$ and $(y_{i})\in \ell_{\infty}(Y)$$)$, then $(A(x_{i}, y_{i}))\in\ell_{r}^{u}(Z)$.
\end{prop}

\begin{proof}
(i)$\Rightarrow$(i$'$). Let $(y_{i})\in \ell_{\infty}(Y)$ and $(x_{i})\in \ell_{q}^{w}(X)$. For all $n\in\mathds{N}$ and all $k\in\mathds{N}$ with $k>n$, we have
\[
\|(A(x_{i}, y_{i}))_{i=n}^{k}\|_{r}^{w} \leq c\,\|(y_{i})_{i=n}^{k}\|_{\infty}\|(x_{i})_{i=n}^{k}\|_{q}^{w}
\]
\[
\leq c\,\|(y_{i})_{i=n}^{\infty}\|_{\infty}\|(x_{i})_{i=n}^{\infty}\|_{q}^{w}.
\]
As this inequality holds for all $k\in\mathds{N}$ with $k>n$, it is clear from the definition of the norm $\|\cdot\|_{r}^{w}$ that
\[
\|(A(x_{i}, y_{i}))_{i=n}^{\infty}\|_{r}^{w} \leq c\,\|(y_{i})_{i=n}^{\infty}\|_{\infty}\|(x_{i})_{i=n}^{\infty}\|_{q}^{w}.
\]

(i$'$)$\Rightarrow$(i). This is trivial.

(ii)$\Rightarrow$(i). Consider the bilinear map
\[
B\,:\,\ell_{q}^{w}(X)\times\ell_{\infty}(Y) \rightarrow \ell_{r}^{w}(Z),
\]
\[
(x_{i},y_{i})_{i}:=((x_{i}),(y_{i})) \mapsto (A(x_{i}, y_{i})).
\]
We shall prove that the graph of $B$ is closed. Let $u_{k} = (x_{i}^{k},y_{i}^{k})_{i}$ be a sequence in $\ell_{q}^{w}(X)\times\ell_{\infty}(Y)$ such that $u_{k}\rightarrow_{k} u=(x_{i},y_{i})_{i}$. Looking at the norms of $\ell_{q}^{w}(X)$ and $\ell_{\infty}(Y)$, we get that $x_{i}^{k}\rightarrow_{k}x_{i}$ and $y_{i}^{k}\rightarrow_{k}y_{i}$ for all $i\in\mathds{N}$. Let $v=(z_{i})\in\ell_{r}^{w}(Z)$ be such that $Bu_{k}\rightarrow_{k} v$, i.e., $(A(x_{i}^{k},y_{i}^{k}))\rightarrow_{k} (z_{i})$. As before, looking at the norm of $\ell_{r}^{w}(Z)$, we get that $A(x_{i}^{k},y_{i}^{k})\rightarrow_{k} z_{i}$ for all $i\in\mathds{N}$. On the other hand, by the continuity of $A$, $A(x_{i}^{k},y_{i}^{k})\rightarrow_{k}A(x_{i},y_{i})$ for all $i\in\mathds{N}$. Hence, $z_{i}=A(x_{i},y_{i})$ for all $i\in\mathds{N}$, i.e., $v=Bu$.

So $B$ has a closed graph. By the closed graph theorem for bilinear maps (see, e.g., \cite[Exercise 1.11, p. 14]{DF} or \cite{F}), $B$ is continuous, hence bounded, and therefore (i) clearly holds.

(iii)$\Rightarrow$(i). We can use the same argument as in (ii)$\Rightarrow$(i).

(i$'$)$\Rightarrow$(ii). Let $(y_{i})\in \ell_{\infty}(Y)$ and $(x_{i})\in \ell_{q}^{w}(X)$. Then $(A(x_{i},y_{i}))\in\ell_{r}^{w}(Z)$, because
\[
\|(A(x_{i}, y_{i}))_{i=1}^{\infty}\|_{r}^{w} \leq c\,\|(y_{i})_{i=1}^{\infty}\|_{\infty}\|(x_{i})_{i=1}^{\infty}\|_{q}^{w} < \infty.
\]

(i$'$)$\Rightarrow$(iii). Let $(y_{i})\in c_{0}(Y)$ and $(x_{i})\in \ell_{q}^{w}(X)$. Then $\|(y_{i})_{i=n}^{\infty}\|_{\infty}\rightarrow 0$ when $n$ tends to infinity and $\|(x_{i})_{i=n}^{\infty}\|_{q}^{w} \leq \|(x_{i})_{i=1}^{\infty}\|_{q}^{w}$. Hence,
\[
\|(A(x_{i}, y_{i}))_{i=n}^{\infty}\|_{r}^{w} \leq c\,\|(y_{i})_{i=n}^{\infty}\|_{\infty}\|(x_{i})_{i=1}^{\infty}\|_{q}^{w} \xrightarrow[n]{} 0,
\]
showing that $(A(x_{i},y_{i}))\in \ell_{r}^{u}(Z)$.

Let now $(y_{i})\in \ell_{\infty}(Y)$ and $(x_{i})\in \ell_{q}^{u}(X)$. Then $\|(y_{i})_{i=n}^{\infty}\|_{\infty} \leq \|(y_{i})_{i=1}^{\infty}\|_{\infty}$ and $\|(x_{i})_{i=n}^{\infty}\|_{q}^{w}\rightarrow 0$ when $n$ tends to infinity. Hence,
\[
\|(A(x_{i},y_{i}))_{i=n}^{\infty}\|_{r}^{w} \leq c\,\|(y_{i})_{i=1}^{\infty}\|_{\infty}\|(x_{i})_{i=n}^{\infty}\|_{q}^{w} \xrightarrow[n]{} 0,
\]
showing that $(A(x_{i},y_{i}))\in \ell_{r}^{u}(Z)$.
\end{proof}

Let $\alpha$ be a tensor norm and let $\mathcal{A}$ be the dual space operator ideal of $\alpha$. As was said in Section \ref{s3}, given an operator $S\in\mathcal{L}(Z,\mathcal{L}(X,Y))$, there always exists an operator $U\in\mathcal{L}(Z\hat{\otimes}_{\pi}X,Y)$ such that $U^{\#}=S$. Being interested in the case when $U\in\mathcal{L}(Z\hat{\otimes}_{\alpha}X,Y)$, we are going to use that
\[
U\in\mathcal{L}(Z\hat{\otimes}_{\alpha}X,Y) \,\,\Leftrightarrow\,\,U^{*}\in\mathcal{L}(Y^{*},\mathcal{A}(Z,X^{*})),
\]
which is straightforward to verify. We shall now apply Theorem \ref{t:3Eve} and the above observation to the special case $\alpha=d_{p}$. Then, as was recalled above, $\mathcal{A}=\mathcal{P}_{p'}=\mathcal{P}_{(p',p')}$.

\begin{thm} \label{t:3equiv}
Let $X$, $Y$, and $Z$ be Banach spaces. Let $1\leq p\leq \infty$. Assume that $S\in\mathcal{L}(Z,\mathcal{L}(X, Y))$. The following statements are equivalent.

$($\emph{a}$)$ There exists an operator $U\in\mathcal{L}(Z\hat{\otimes}_{d_{p}}X,Y)$ such that $U^{\#} = S$.

$($\emph{b}$)$ There exists a constant $c>0$ such that, for all finite systems $(x_{i})_{i=1}^{n}\subset X$ and $(z_{i})_{i=1}^{n}\subset Z$,
\[
\|((Sz_{i})x_{i})\|_{p'}^{w} \leq c\,\|(x_{i})\|_{\infty}\|(z_{i})\|_{p'}^{w}.
\]

$($\emph{b$'$}$)$ There exists a constant $c>0$ such that, for all $(x_{i})\in \ell_{\infty}(X)$ and $(z_{i})\in \ell_{p'}^{w}(Z)$, and for all $n\in\mathds{N}$,
\[
\|((Sz_{i})x_{i})_{i=n}^{\infty}\|_{p'}^{w} \leq c\,\|(x_{i})_{i=n}^{\infty}\|_{\infty}\|(z_{i})_{i=n}^{\infty}\|_{p'}^{w}.
\]

$($\emph{c}$)$ If $(x_{i})\in \ell_{\infty}(X)$ and $(z_{i})\in \ell_{p'}^{w}(Z)$, then $((Sz_{i})x_{i})\in\ell_{p'}^{w}(Y)$.

$($\emph{d}$)$ If $(x_{i})\in c_{0}(X)$ and $(z_{i})\in \ell_{p'}^{w}(Z)$ $($or $(x_{i})\in \ell_{\infty}(X)$ and $(z_{i})\in \ell_{p'}^{u}(Z)$$)$, then $((Sz_{i})x_{i})\in\ell_{p'}^{u}(Y)$.
\end{thm}

\begin{proof}
Recall that $S$ is associated to some $U\in\mathcal{L}(Z\hat{\otimes}_{\pi}X,Y)$, meaning that $U(z\otimes x) = (Sz)x$ for all $z\in Z$ and $x\in X$.

By the above observation, condition (a) is equivalent to the fact that $U^{*}\in\mathcal{L}(Y^{*},\mathcal{P}_{p'}(Z,X^{*}))$. Therefore, viewing $U^{*}$ in the role of $T$ in Theorem \ref{t:3Eve}, condition (a) is equivalent to the existence of a constant $c>0$ such that, for all finite systems $(x_{i}^{**})_{i=1}^{n}\subset X^{**}$ and $(z_{i})_{i=1}^{n}\subset Z$,
\begin{equation} \label{eq:b}
\|(U^{**}(z_{i}\otimes x_{i}^{**}))\|_{p'}^{w} \leq c\,\|(x_{i}^{**})\|_{\infty}\|(z_{i})\|_{p'}^{w}.
\end{equation}

Now, an easy application of the principle of local reflexivity yields that the above condition is equivalent to the existence of a constant $C>0$ such that, for all finite systems $(x_{i})_{i=1}^{n}\subset X$ and $(z_{i})_{i=1}^{n}\subset Z$,
\begin{equation} \label{eq:b'}
\|(U(z_{i}\otimes x_{i}))\|_{p'}^{w} \leq C\,\|(x_{i})\|_{\infty}\|(z_{i})\|_{p'}^{w},
\end{equation}
which in turn is clearly equivalent to (b).

For completeness, let us include the proof that the conditions concerning \eqref{eq:b} and \eqref{eq:b'} are equivalent. Since $U^{**}$ is an extension of $U$, \eqref{eq:b} implies \eqref{eq:b'}. For the converse, let $(x_{i}^{**})_{i=1}^{n}\subset X^{**}$ and $(z_{i})_{i=1}^{n}\subset Z$. Fix an arbitrary $y^{*}\in B_{Y^{*}}$. Considering $\mbox{span}\{x_{i}^{**}\}_{i=1}^{n}$ in $X^{**}$ and $\mbox{span}\{(U^{*}y^{*})z_{i}\}_{i=1}^{n}$ in $X^{*}$, the principle of local reflexivity yields an operator $V: \mbox{span}\{x_{i}^{**}\}_{i=1}^{n}\rightarrow X$, with $\|V\|\leq 2$, such that $\langle Vx_{i}^{**},(U^{*}y^{*})z_{i}\rangle = \langle (U^{*}y^{*})z_{i},x_{i}^{**}\rangle$ for all $i=1,\ldots,n$. But then $y^{*}(U(z_{i}\otimes Vx_{i}^{**})) = \langle U^{*}y^{*}, z_{i}\otimes Vx_{i}^{**} \rangle = \langle  Vx_{i}^{**}, (U^{*}y^{*})z_{i} \rangle = \langle (U^{*}y^{*})z_{i}, x_{i}^{**} \rangle$ for all $i=1,\ldots,n$. Therefore,
\[
\Big( \sum_{i=1}^{n} |\langle  y^{*},U^{**}(z_{i}\otimes x_{i}^{**})\rangle|^{p'}\Big)^{1/p'} = \Big( \sum_{i=1}^{n} |\langle(U^{*}y^{*})z_{i},x_{i}^{**}\rangle|^{p'}\Big)^{1/p'}
\]
\[
= \Big( \sum_{i=1}^{n} | y^{*}(U(z_{i}\otimes Vx_{i}^{**})) |^{p'}\Big)^{1/p'} \leq \|(U(z_{i}\otimes Vx_{i}^{**}))\|_{p'}^{w} \leq 2\,C\,\|(x_{i}^{**})\|_{\infty}\|(z_{i})\|_{p'}^{w}.
\]
Taking the supremum over $y^{*}\in B_{Y^{*}}$ gives us \eqref{eq:b} (with $c=2C$).

The equivalences of (b), (b$'$), (c), and (d) are clear from Proposition \ref{p:t3Eve}.
\end{proof}

Recalling that $\mathcal{C}_{p}(\Omega,X) = \mathcal{C}(\Omega)\hat{\otimes}_{d_{p}}X$, let us spell out the desired (immediate) consequence of Theorem \ref{t:3equiv}.

\begin{cor} \label{c:4equiv}
Let $X$ and $Y$ be Banach spaces and let $\Omega$ be a compact Hausdorff space. Let $1\leq p\leq\infty$. Assume that $S\in\mathcal{L}(\mathcal{C}(\Omega), \mathcal{L}(X, Y))$. The following statements are equivalent.

$($\emph{a}$)$ There exists an operator $U\in\mathcal{L}(\mathcal{C}_{p}(\Omega,X),Y)$ such that $U^{\#} = S$.

$($\emph{b}$)$ There exists a constant $c>0$ such that, for all finite systems $(x_{i})_{i=1}^{n}\subset X$ and $(\varphi_{i})_{i=1}^{n}\subset \mathcal{C}(\Omega)$,
\[
\|((S\varphi_{i})x_{i})\|_{p'}^{w} \leq c\,\|(x_{i})\|_{\infty}\|(\varphi_{i})\|_{p'}^{w}.
\]

$($\emph{b$'$}$)$ There exists a constant $c>0$ such that, for all $(x_{i})\in \ell_{\infty}(X)$ and $(\varphi_{i})\in \ell_{p'}^{w}(\mathcal{C}(\Omega))$, and for all $n\in\mathds{N}$,
\[
\|((S\varphi_{i})x_{i})_{i=n}^{\infty}\|_{p'}^{w} \leq c\,\|(x_{i})_{i=n}^{\infty}\|_{\infty}\|(\varphi_{i})_{i=n}^{\infty}\|_{p'}^{w}.
\]

$($\emph{c}$)$ If $(x_{i})\in \ell_{\infty}(X)$ and $(\varphi_{i})\in \ell_{p'}^{w}(\mathcal{C}(\Omega))$, then $((S\varphi_{i})x_{i})\in\ell_{p'}^{w}(Y)$.

$($\emph{d}$)$ If $(x_{i})\in c_{0}(X)$ and $(\varphi_{i})\in \ell_{p'}^{w}(\mathcal{C}(\Omega))$ $($or $(x_{i})\in \ell_{\infty}(X)$ and $(\varphi_{i})\in \ell_{p'}^{u}(\mathcal{C}(\Omega))$$)$, then $((S\varphi_{i})x_{i})\in\ell_{p'}^{u}(Y)$.
\end{cor}

For the case $p=\infty$, and hence $p'=1$, recall from Section \ref{s3} that $\mathcal{C}_{\infty}(\Omega, X) = \mathcal{C}(\Omega,X)$ as Banach spaces. Thus, the above corollary also characterizes those operators $S\in\mathcal{L}(\mathcal{C}(\Omega),\mathcal{L}(X, Y))$ which are associated to an operator $U\in\mathcal{L}(\mathcal{C}(\Omega,X), Y)$. This develops further Dinculeanu's remark (see the Introduction) that such an operator $U$ does not necessarily exist for a given arbitrary $S$.

Another consequence of Theorem \ref{t:3equiv} is the next result.

\begin{cor}\label{Uexist}
Let $X$, $Y$, and $Z$ be Banach spaces. Let $1\leq p\leq\infty$. If $S\in\mathcal{P}_{p'}(Z,\mathcal{L}(X, Y))$, then there exists an operator $U\in\mathcal{L}(Z\hat{\otimes}_{d_{p}}X,Y)$ such that $U^{\#} = S$.
\end{cor}

\begin{proof}
Given $(x_i)\in\ell_{\infty}(X)$ and $(z_i)\in\ell_{p'}^{w}(Z)$, we have
\[
\|((Sz_{i})x_{i})\|_{p'}^{w} \leq \|((Sz_{i})x_{i})\|_{p'} \leq \|(x_{i})\|_{\infty} \|(Sz_{i})\|_{p'} \leq \|S\|_{\mathcal{P}_{p'}}\|(x_{i})\|_{\infty}\|(z_{i})\|_{p'}^{w}.
\]
\end{proof}

In Section \ref{s4}, we shall show that the converse of this result, in general, does not hold (see Example \ref{ex:1}).

\section{Characterizing associated operators: classical case}\label{s5}

Let $X$ and $Y$ be Banach spaces and let $\Omega$ be a compact Hausdorff space. Let $1 \leq p \leq \infty$. In Sections \ref{s3} and \ref{s2}, for a given operator $S\in\mathcal{L}(\mathcal{C}(\Omega), \mathcal{L}(X,Y))$, we obtained conditions for the existence of an operator $U\in\mathcal{L}(\mathcal{C}_{p}(\Omega,X),Y)$ such that $U^{\#} = S$. This was done in a rather general framework involving tensor products of Banach spaces. In this section, we are interested in specific conditions involving a representing measure of $S$. In terms of the representing measure of $S$, we establish a necessary condition for the existence of such an operator $U$ (Proposition \ref{p:condnecess}) that becomes also sufficient in the classical case of $\mathcal{C}(\Omega,X) = \mathcal{C}_{\infty}(\Omega,X)$ (Proposition \ref{p:condnecesuf}).

We denote by $\Sigma$ the $\sigma$-algebra of Borel subsets of $\Omega$. The space of $\Sigma$-simple functions with values in $X$ and the space of all bounded $\Sigma$-measurable functions with values in $X$ (i.e., the space of all functions from $\Omega$ into $X$ which are the uniform limit of a sequence of $\Sigma$-simple functions) are denoted by $\mathcal{S}(\Sigma,X)$ and $\mathcal{B}(\Sigma,X)$, respectively. In case $X=\mathds{K}$, we abbreviate them to $\mathcal{S}(\Sigma)$ and $\mathcal{B}(\Sigma)$, respectively.

It is well known that, for every operator $S\in\mathcal{L}(\mathcal{C}(\Omega),Y)$, there exists a vector measure $m:\Sigma\rightarrow Y^{**}$ of bounded semivariation such that
\[
S\varphi = \int_{\Omega}\varphi\,dm, \,\,\varphi\in\mathcal{C}(\Omega),
\]
(see, e.g., \cite[Theorem 1, p. 152]{DU}). The vector measure $m$ is called the \emph{representing measure} of $S$. We extend this result from $Y\cong\mathcal{L}(\mathds{K},Y)$ to $\mathcal{L}(X,Y)$: in the case when $S\in\mathcal{L}(\mathcal{C}(\Omega),\mathcal{L}(X,Y))$, we build a representing measure which takes its values in $\mathcal{L}(X,Y^{**})$ as follows.

So, let $S\in\mathcal{L}(\mathcal{C}(\Omega),\mathcal{L}(X,Y))$. For each $x\in X$, we define an operator $S_{x}\in\mathcal{L}(\mathcal{C}(\Omega),Y)$ by $S_{x}\varphi = (S\varphi)x$, $\varphi\in\mathcal{C}(\Omega)$. Let $m_{x}:\Sigma\rightarrow Y^{**}$ be the representing measure of $S_{x}$. We define
\[
m:\Sigma\rightarrow \mathcal{L}(X,Y^{**}),
\]
\[
E\mapsto m(E),
\]
by
\begin{equation} \label{measureS}
\langle y^{*},m(E)x\rangle = \langle y^{*},m_{x}(E)\rangle
\end{equation}
for all $x\in X$ and $y^{*}\in Y^{*}$. Using that $\|m_{x}\|(\Omega) = \|S_{x}\|$ (see, e.g., \cite[Theorem 1, p. 152]{DU}), a straightforward verification shows that the set map $m$ is a finitely additive vector measure and its range is bounded by $\|S\|$. Hence, $m$ is of bounded semivariation (see, e.g., \cite[Proposition 11, p. 4]{DU}), and, as we shall see soon (Proposition \ref{mofS}), $\|m\|(\Omega)=\|S\|$.

We can connect the integral with respect to $m$ with the integral with respect to $m_{x}$ as follows:
\begin{equation} \label{measurelation2}
\langle y^{*},\Big(\int_{\Omega}\varphi\, dm\Big)x\rangle= \langle y^{*},\int_{\Omega}\varphi\, dm_{x} \rangle
\end{equation}
for all $\varphi\in \mathcal{B}(\Sigma)$, $x\in X$ and $y^{*}\in Y^{*}$. (This equality follows easily from \eqref{measureS} using a standard argument which passes from characteristic functions to functions in $\mathcal{S}(\Sigma)$ by linearity, and finally to functions in $\mathcal{B}(\Sigma)$ by density.).

In particular, \eqref{measurelation2} is also true for $\varphi\in\mathcal{C}(\Omega)$. In this case, $\int_{\Omega}\varphi\,dm_{x} = S_{x}\varphi$, because $m_{x}$ is the representing measure of $S_{x}$. Thus
\[
\langle y^{*},\Big(\int_{\Omega}\varphi\, dm\Big)x\rangle = \langle S_{x}\varphi,y^{*}\rangle = \langle (S\varphi)x, y^{*}\rangle
\]
for all $\varphi\in\mathcal{C}(\Omega)$, $x\in X$, and $y^{*}\in Y^{*}$. Then,
\[
S\varphi = \int_{\Omega}\varphi\, dm
\]
for all $\varphi\in\mathcal{C}(\Omega)$, showing that $m$ is a representing measure of $S$. The above integral is the restriction to $\mathcal{C}(\Omega)$ of the elementary Bartle integral $\int_{\Omega}(\cdot)dm$ defined on $\mathcal{B}(\Sigma)$.

The following result, which is of independent interest, is now easy to deduce.

\begin{prop}\label{mofS}
Let $X$ and $Y$ be Banach spaces and let $\Omega$ be a compact Hausdorff space. Assume that $S\in\mathcal{L}(\mathcal{C}(\Omega),\mathcal{L}(X, Y))$ and let $m:\Sigma\rightarrow \mathcal{L}(X, Y^{**})$ be its representing measure. Then $\|m\|(\Omega)=\|S\|$.
\end{prop}

\begin{proof}
Denoting by $\hat{S}\in\mathcal{L}(\mathcal{B}(\Sigma), \mathcal{L}(X, Y^{**}))$ the integration operator $\hat{S}\varphi = \int_{\Omega}\varphi\,dm$, $\varphi\in\mathcal{B}(\Sigma)$, we see that $\hat{S}$ extends $S$ (here, as usual, we identify $Y$ with a subspace of $Y^{**}$). Hence,
\[
\|S\| \leq \|\hat{S}\| = \|m\|(\Omega);
\]
for the last equality see, e.g., \cite[Theorem 13, p. 6]{DU}.

On the other hand, when we look at $\mathcal{B}(\Sigma)$ as a closed subspace of $\mathcal{C}(\Omega)^{**}$, then $\hat{S}$ is the restriction to $\mathcal{B}(\Sigma)$ of $S^{**}\in\mathcal{L}(\mathcal{C}(\Omega)^{**}, \mathcal{L}(X,Y^{**})^{**})$ (this is a standard argument; see, e.g., \cite[pp. 152--153]{DU}). Hence,
\[
\|m\|(\Omega) = \|\hat{S}\| \leq \|S^{**}\| = \|S\|.
\]
\end{proof}

Let us present now the promised necessary condition announced at the beginning of this section. We denote by $\chi_{E}$ the \emph{characteristic function} of $E\in\Sigma$.

\begin{prop} \label{p:condnecess}
Let $X$ and $Y$ be Banach spaces and let $\Omega$ be a compact Hausdorff space. Assume that $S\in\mathcal{L}(\mathcal{C}(\Omega),\mathcal{L}(X, Y))$ and let $m:\Sigma\rightarrow \mathcal{L}(X, Y^{**})$ be its representing measure. Let $1\leq p\leq\infty$. If there exists an operator $U\in \mathcal{L}(\mathcal{C}_{p}(\Omega, X),Y)$ such that $U^{\#}=S$, then there exists a constant $c>0$ such that, for every sequence $(E_{i})$ of pairwise disjoint sets in $\Sigma$ and every sequence $(x_{i})$ in $B_{X}$, the inequality $\|(m(E_{i})x_{i})\|_{p'}^{w}\leq c$ holds.
\end{prop}

\begin{proof}
Let $S\in\mathcal{L}(\mathcal{C}(\Omega),\mathcal{L}(X, Y))$ and let $m:\Sigma\rightarrow \mathcal{L}(X, Y^{**})$ be its representing measure. As in the proof of Proposition \ref{mofS}, we define $\hat{S}\in\mathcal{L}(\mathcal{B}(\Sigma),\mathcal{L}(X, Y^{**}))$ by $\hat{S}\varphi = \int_{\Omega}\varphi\,dm$, $\varphi\in \mathcal{B}(\Sigma)$. Then $\hat{S}|_{\mathcal{C}(\Omega)} = S$ and, by \eqref{measurelation2}, recalling that $\int_{\Omega}\varphi\,dm_{x} = S_{x}^{**}\varphi$, $\varphi\in\mathcal{B}(\Sigma)$ (where $S_{x}\in\mathcal{L}(\mathcal{C}(\Omega),Y)$ is defined by $S_{x}\varphi = (S\varphi)x$, $\varphi\in\mathcal{C}(\Omega)$, $x\in X$), we have
\begin{equation} \label{eq:prop1}
\langle y^{*}, (\hat{S}\varphi)x \rangle = \langle S_{x}^{*}y^{*},\varphi \rangle
\end{equation}
for all $\varphi\in\mathcal{B}(\Sigma)\subset \mathcal{C}(\Omega)^{**}$, $x\in X$, and $y^{*}\in Y^{*}$.

Let $U\in \mathcal{L}(\mathcal{C}_{p}(\Omega, X),Y)$ be such that $U^{\#}=S$. We can define $\hat{U}\in\mathcal{L}(\mathcal{B}(\Sigma)\hat{\otimes}_{d_{p}}X,Y^{**})$ by $\hat{U}(\varphi\otimes x) = (\hat{S}\varphi)x$. Indeed, we already know that $\hat{U}:\mathcal{B}(\Sigma)\otimes_{d_{p}}X\rightarrow Y^{**}$ is a linear operator. It suffices to show that $\hat{U}$ is bounded on $\mathcal{B}(\Sigma)\otimes_{d_{p}}X$. Recall that $\mathcal{C}_{p}(\Omega,X) = \mathcal{C}(\Omega)\hat{\otimes}_{d_{p}}X$. Let $\varphi\in\mathcal{C}(\Omega)$, $x\in X$, and $y^{*}\in Y^{*}$. Using that $U(\varphi\otimes x)=S_{x}\varphi$ in $Y$, hence $\langle \varphi\otimes x, U^{*}y^{*}\rangle = \langle \varphi , S_{x}^{*}y^{*}\rangle$, and using also that $(\mathcal{C}(\Omega)\hat{\otimes}_{d_{p}}X)^{*} = \mathcal{P}_{p'}(\mathcal{C}(\Omega),X^{*})$ as Banach spaces, we get that $(U^{*}y^{*})^{*}x = S_{x}^{*}y^{*}$ in $\mathcal{C}(\Omega)^{*}$. Since $\mathcal{B}(\Sigma)\subset\mathcal{C}(\Omega)^{**}$, for all $\varphi\in\mathcal{B}(\Sigma)$, $x\in X$, and $y^{*}\in Y^{*}$, we have
\begin{equation} \label{eq:prop2}
\langle S_{x}^{*}y^{*},\varphi\rangle = \langle x, (U^{*}y^{*})^{**}\varphi\rangle = \langle \varphi\otimes x, U_{y^{*}}\rangle,
\end{equation}
where $U_{y^{*}}:=(U^{*}y^{*})^{**}|_{\mathcal{B}(\Sigma)}$. It is well known (see, e.g., \cite[Theorem 2.17 and 2.19]{DJT}) that $(U^{*}y^{*})^{**}\in\mathcal{P}_{p'}(\mathcal{C}(\Omega)^{**},X^{*})$ and $\|(U^{*}y^{*})^{**}\|_{\mathcal{P}_{p'}} = \|U^{*}y^{*}\|_{\mathcal{P}_{p'}}$. It follows that $U_{y^{*}}\in\mathcal{P}_{p'}(\mathcal{B}(\Sigma),X^{*})$ and  $\|U_{y^{*}}\|_{\mathcal{P}_{p'}} = \|U^{*}y^{*}\|_{\mathcal{P}_{p'}}$, because $\mathcal{C}(\Omega) \subset \mathcal{B}(\Sigma) \subset \mathcal{C}(\Omega)^{**}$ as closed subspaces. Recalling that $\mathcal{P}_{p'}(\mathcal{B}(\Sigma),X^{*}) = (\mathcal{B}(\Sigma)\otimes_{d_{p}} X)^{*}$ and using \eqref{eq:prop1} and \eqref{eq:prop2}, we may write
\[
\langle y^{*},\hat{U}(\varphi\otimes x)\rangle = \langle \varphi\otimes x, U_{y^{*}}\rangle, \ \ \varphi\in\mathcal{B}(\Sigma), x\in X, y^{*}\in Y^{*}.
\]
Hence, for any $v=\sum_{i=1}^{n}\varphi_{i}\otimes x_{i}\in\mathcal{B}(\Sigma)\otimes_{d_{p}}X$ and $y^{*}\in Y^{*}$, we get that
\[
|\langle y^{*},\hat{U}v\rangle | = \Big|\sum_{i=1}^{n}\langle \varphi_{i}\otimes x_{i}, U_{y^{*}}\rangle \Big| = |\langle v, U_{y^{*}}\rangle| \leq \|U_{y^{*}}\|_{\mathcal{P}_{p'}}\|v\|_{d_{p}}.
\]
But since
\[
\|U_{y^{*}}\|_{\mathcal{P}_{p'}} = \|U^{*}y^{*}\|_{\mathcal{P}_{p'}} = \sup \Big\{ |\langle v, U^{*}y^{*}\rangle|\,:\,v\in B_{\mathcal{C}(\Omega)\otimes_{d_{p}}X}\Big\}
\]
\[
= \sup \Big\{ |\langle Uv, y^{*}\rangle|\,:\,v\in B_{\mathcal{C}(\Omega)\otimes_{d_{p}}X}\Big\} \leq \|U\|\,\|y^{*}\|,
\]
$\hat{U}$ is bounded on $\mathcal{B}(\Sigma)\otimes_{d_{p}}X$, as desired.

We have defined $\hat{U}$ so that $\hat{U}^{\#} = \hat{S}$. Hence, by Theorem \ref{t:3equiv}, there exists a constant $c>0$ such that, for all $(x_{i})\in\ell_{\infty}(X)$ and $(\varphi_{i})\in\ell_{p'}^{w}(\mathcal{B}(\Sigma))$,
\begin{equation} \label{eq:gen}
\|((\hat{S}\varphi_{i})x_{i})\|_{p'}^{w} \leq c\,\|(x_{i})\|_{\infty}\|(\varphi_{i})\|_{p'}^{w}.
\end{equation}

Let $(E_{i})$ be a sequence of pairwise disjoint sets in $\Sigma$. Then $(\chi_{E_{i}})\in\ell_{p'}^{w}(\mathcal{B}(\Sigma))$ and $\|(\chi_{E_{i}})\|_{p'}^{w}\leq1$. Indeed, let us show that even $\|(\chi_{E_{i}})\|_{1}^{w}\leq 1$. By definition,
\[
\|(\chi_{E_{i}})\|_{1}^{w} = \sup\Big\{ \|(\langle \chi_{E_{i}},\mu\rangle )_{i}||_{1}\,:\,\mu\in B_{\mathcal{B}(\Sigma)^{*}}\Big\}.
\]
Also, it is well known that $\mathcal{B}(\Sigma)^{*} = ba(\Sigma)$ as Banach spaces, where $ba(\Sigma)$ denotes the space of all bounded additive measures defined on $\Sigma$ with the variation norm, and the duality works as follows:
\[
\langle \varphi,\mu\rangle = \int_{\Omega} \varphi\,d\mu \,\, ,\varphi\in\mathcal{B}(\Sigma), \mu\in ba(\Sigma)
\]
(see, e.g., \cite[Theorem 1, p. 258]{DS}). Since, in particular,
\[
\langle \chi_{E},\mu\rangle = \mu(E), \mbox{ for all } E\in\Sigma,
\]
we have
\[
\|(\chi_{E_{i}})\|_{1}^{w} = \sup\Big\{ \|(\mu(E_{i}))_{i}||_{1}\,:\,\mu\in B_{ba(\Sigma)}\Big\}.
\]
But
\[
\|(\mu(E_{i}))_{i}||_{1} \leq |\mu|(\bigcup_{i=1}^{\infty}E_{i}) \leq |\mu|(\Omega),
\]
giving that $\|(\chi_{E_{i}})\|_{1}^{w} \leq 1$.

Therefore, since $\hat{S}\chi_{E_{i}} = m(E_{i})$, we get from \eqref{eq:gen} the desired inequality for $(x_{i})\subset B_{X}$.
\end{proof}

\begin{rmk} \label{r:4.3}
Using the general inequality $\|z_{i}\|\leq \|(z_{i})_{i}\|_{1}^{w}$ for all $i$, it is clear that $\|(\chi_{E_{i}})\|_{1}^{w} = 1$ if there exists $E_{i_{0}}\neq \emptyset$. An alternative proof of the inequality $\|(\chi_{E_{i}})\|_{1}^{w} \leq 1$ can be done using extreme points. It is well known that we only need to use any norming subset of $\mathcal{B}(\Sigma)^{*}$ for evaluating $\|(\chi_{E_{i}})\|_{1}^{w}$ and a particularly useful example of this kind is the set of all extreme points of $B_{\mathcal{B}(\Sigma)^{*}}$ (see, e.g., \cite[p. 36]{DJT}). This set coincides with $\{\alpha \delta_{\omega} : \alpha\in\mathds{K}, |\alpha|=1, \omega\in\Omega\}$, where $\delta_{\omega}\in \mathcal{B}(\Sigma)^{*}$ is defined by $\langle \varphi, \delta_{\omega} \rangle = \varphi(\omega)$, $\omega \in \Omega$ (see, e.g., \cite[Chapter V, Theorem 8.4]{C}). Then
\[
\|(\chi_{E_{i}})\|_{1}^{w} =  \sup \Big\{ \| (\langle \chi_{E_{i}}, \alpha \delta_{\omega}\rangle)_{i}\|_{1} : \alpha\in\mathds{K}, |\alpha|=1, \omega\in\Omega\ \Big\}
\]
\[
= \sup \Big\{  \|(\alpha \chi_{E_{i}}(\omega))_{i}\|_{1} : \alpha\in\mathds{K}, |\alpha|=1, \omega\in\Omega\ \Big\}
\]
\[
= \sup \Big\{  \|(\chi_{E_{i}}(\omega))_{i}\|_{1} : \omega\in\Omega\ \Big\} \leq 1,
\]
where the last inequality uses that $(E_{i})$ is a sequence of pairwise disjoint sets in $\Sigma$ (this inequality becomes an equality if there exists $E_{i_{0}}\neq \emptyset$).
\end{rmk}

In Section \ref{s4}, we shall see that the necessary condition from Proposition \ref{p:condnecess}, in general, is not sufficient (see Example \ref{ex:2}). But it is so for the case $p=\infty$, as the next proposition shows.

\begin{prop} \label{p:condnecesuf}
Let $X$ and $Y$ be Banach spaces and let $\Omega$ be a compact Hausdorff space. Assume that $S\in\mathcal{L}(\mathcal{C}(\Omega),\mathcal{L}(X, Y))$ and let $m:\Sigma\rightarrow \mathcal{L}(X, Y^{**})$ be its representing measure. The following statements are equivalent.

$($\emph{a}$)$ There exists an operator $U\in\mathcal{L}(\mathcal{C}(\Omega, X),Y)$ such that $U^{\#}=S$.

$($\emph{b}$)$ There exists a constant $c>0$ such that, for every sequence $(E_{i})$ of pairwise disjoint sets in $\Sigma$ and every sequence $(x_{i})$ in $B_{X}$, the inequality $\|(m(E_{i})x_{i})\|_{1}^{w}\leq c$ holds.
\end{prop}

\begin{proof}
We only need to prove (b)$\Rightarrow$(a). Similarly to the construction of the elementary Bartle integral (the case when $X=\mathds{K}$) (see, e.g., \cite[pp. 5--6]{DU}), we have a linear operator $V:\mathcal{S}(\Sigma,X)\rightarrow Y^{**}$ defined by $Vf = \sum_{i=1}^{n}m(E_{i})x_{i}$ for $f=\sum_{i=1}^{n}\chi_{E_{i}}x_{i}$, where $E_{1},\ldots,E_{n}$ are pairwise disjoint sets in $\Sigma$ and $x_{i}\in X\setminus\{0\}$. Using that $\|f\| = \max_{i}\|x_{i}\|$, we have that
\[
\|Vf\| = \sup \Big\{ \Big|\langle y^{*},\sum_{i=1}^{n}m(E_{i})x_{i}\rangle\Big| \, : \, \|y^{*}\|\leq1 \Big\}
\]
\[
\leq \sup \Big\{ \sum_{i=1}^{n}\Big|\langle y^{*},m(E_{i})\frac{x_{i}}{\|x_{i}\|}\rangle\Big|\,\|x_{i}\| \, : \, \|y^{*}\|\leq1 \Big\}
\]
\[
\leq \sup \Big\{ \sum_{i=1}^{n} \Big|\langle y^{*},m(E_{i})\frac{x_{i}}{\|x_{i}\|}\rangle\Big| \, : \, \|y^{*}\|\leq1 \Big\} \max_{i}\|x_{i}\|
\]
\[
= \Big\| \Big( m(E_{i})\Big( \frac{x_{i}}{\|x_{i}\|} \Big)\Big) \Big\|_{1}^{w} \|f\|  \leq c\,\|f\|.
\]
Therefore, $V$ is continuous and it can be extended by continuity to $\mathcal{B}(\Sigma,X)$. We keep calling its extension $V$. We shall show that $U = V|_{\mathcal{C}(\Omega, X)}$ is the desired operator.

As in the proof of Proposition \ref{p:condnecess}, we have $\hat{S}\in\mathcal{L}(\mathcal{B}(\Sigma),\mathcal{L}(X,Y^{**}))$ defined by $\hat{S}\varphi = \int_{\Omega}\varphi\,dm$, $\varphi\in \mathcal{B}(\Sigma)$. Let $x\in X$. Since
\[
(\hat{S}\chi_{E})x = m(E)x = V(\chi_{E}x), \ E\in\Sigma,
\]
and $\mathcal{B}(\Sigma) = \overline{\mathcal{S}(\Sigma)}$,
\[
(\hat{S}\varphi)x = V(\varphi x), \ \varphi\in\mathcal{B}(\Sigma), \ x\in X.
\]
In particular,
\[
(S\varphi)x = U(\varphi x), \ \varphi\in\mathcal{C}(\Omega), \ x\in X,
\]
meaning that $U^{\#} = S$. It remains to show that ${\rm ran}\,U \subset Y$. From the last equality, we see that $U(\varphi x)\in Y$ for all $\varphi\in\mathcal{C}(\Omega)$ and $x\in X$. But, as is well known,
\[
\mathcal{C}(\Omega, X) = \overline{\mbox{span}} \{\varphi x\,:\,\varphi\in\mathcal{C}(\Omega), x\in X \}
\]
(see, e.g., \cite[p. 49]{R}; this is, in fact, the main argument in the proof of Grothendieck's description $\mathcal{C}(\Omega, X) = \mathcal{C}(\Omega)\hat{\otimes}_{\varepsilon}X)$.
\end{proof}

\section{Examples} \label{s4}

In Corollary \ref{p:cotype2}, we proved that, for all Banach spaces $X$ and $Y$, with $X^{*}$ being of cotype 2, and for every $S\in\mathcal{L}(\mathcal{C}(\Omega),\mathcal{L}(X, Y))$ and $p\leq2$, there exists an operator $U\in\mathcal{L}(\mathcal{C}_{p}(\Omega, X),Y)$ such that $U^{\#}=S$. The next example proves that this result does not hold for $p>2$.

\begin{exam} \label{ex:cotype2NO}
There exists an operator $S\in\mathcal{L}(\mathcal{C}([0,1]),\mathcal{L}(\ell_{2},\ell_{2}))$ such that, for each $p>2$, there does not exist any operator $U\in\mathcal{L}(\mathcal{C}_{p}([0,1],\ell_{2}),\ell_{2})$ such that $U^{\#} = S$.

\begin{proof}
Consider an operator $S\in\mathcal{L}(\mathcal{C}([0,1]),\mathcal{L}(\ell_{2},\ell_{2}))$ defined by $(S\varphi)x = (x(n)\varphi(\delta_{n}))_{n}$, for $x=(x(n))_{n}\in\ell_{2}$, where $(\delta_{n})_{n=1}^{\infty}$ is the following sequence in $[0,1]$:
\[
\frac{1}{2}; \frac{1}{2^{2}}, \frac{3}{2^{2}}; \frac{1}{2^{3}}, \frac{3}{2^{3}}, \frac{5}{2^{3}}, \frac{7}{2^{3}}; \cdots; \frac{1}{2^{n}}, \frac{3}{2^{n}}, \cdots, \frac{2^{n}-1}{2^{n}}; \cdots
\]

By contradiction, let us suppose that there exist $p>2$ and $U\in\mathcal{L}(\mathcal{C}_{p}([0,1],\ell_{2}),\ell_{2})$ such that $U^{\#} = S$. By Theorem \ref{t:3equiv}, there exists a constant $c>0$ such that $\|((S\varphi_{n})x_{n})\|_{p'}^{w}\leq c$, for all $(x_{n})\subset B_{\ell_{2}}$ and $(\varphi_{n})\subset\mathcal{C}([0,1])$ with $\|(\varphi_{n})\|_{p'}^{w}\leq1$. Taking $(x_{n})=(e_{n})$, the unit vector basis in $\ell_{2}$, we get that $(S\varphi_{n})e_{n} = \varphi_{n}(\delta_{n})e_{n}$, $n\in\mathds{N}$. Therefore, $\sum_{n=1}^{\infty}|\langle \varphi_{n}(\delta_{n})e_{n}, y\rangle|^{p'}\leq c^{p'}$, for all $y=(y(n))\in B_{\ell_{2}}$, meaning that
\begin{equation} \label{eq:1}
\sum_{n=1}^{\infty}| \varphi_{n}(\delta_{n})y(n)|^{p'}\leq c^{p'}
\end{equation}
for all $y=(y(n))\in B_{\ell_{2}}$ and ($\varphi_{n})\subset\mathcal{C}([0,1])$ with $\|(\varphi_{n})\|_{p'}^{w}\leq1$.

Let us consider the functions of the classical Faber--Schauder basis, except the two first ones. Here are the first few of them:

  \begin{center}
    \begin{tabular}{c}
      \scalebox{0.4}{\includegraphics{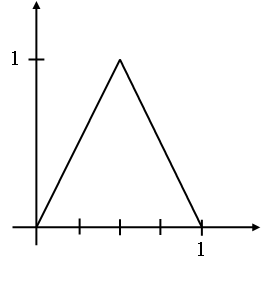}}
      \\
      $\varphi_{1}^{1}$
    \end{tabular}
    \\
    \begin{tabular}{c c c}
      \scalebox{0.4}{\includegraphics{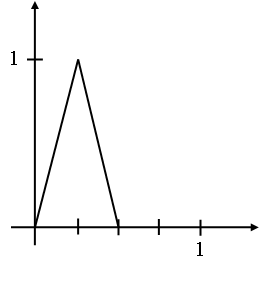}}
      &
      \ \ \ \ \ \
      &
      \scalebox{0.4}{\includegraphics{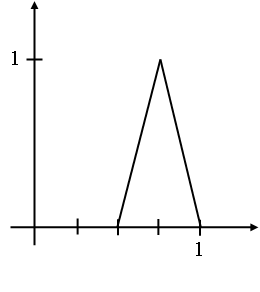}}
      \\
      $\varphi_{1}^{2}$
      &
      \,\,\,\,\,\,\,\,
      &
      $\varphi_{2}^{2}$
    \end{tabular}
    \\
    \begin{tabular}{c c c c}
      \scalebox{0.4}{\includegraphics{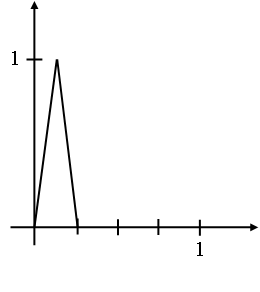}}
      &
      \scalebox{0.4}{\includegraphics{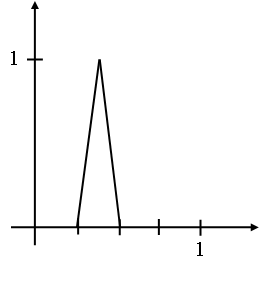}}
      &
      \scalebox{0.4}{\includegraphics{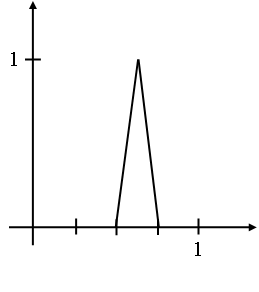}}
      &
      \scalebox{0.4}{\includegraphics{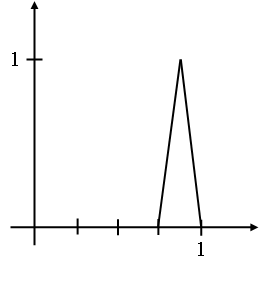}}
      \\
      $\varphi_{1}^{3}$
      &
      $\varphi_{2}^{3}$
      &
      $\varphi_{3}^{3}$
      &
      $\varphi_{4}^{3}$
    \end{tabular}
  \end{center}
In general, given $k\in\mathds{N}$, we take $\varphi_{n}^{k}$ with $n=1,2,\ldots,2^{k-1}$ defined by
\[
\varphi_{n}^{k}(t) = \left\{\begin{array}{ll}
1 - | 2^{k}t - 2n + 1|&\mbox{ if }\,\,\frac{2n-2}{2^{k}} \leq t \leq \frac{2n}{2^{k}},\\
0&\mbox{ otherwise.}\end{array}\right.
\]

Notice that $\varphi_{1}^{1}(\delta_{1})=1, \varphi_{1}^{2}(\delta_{2})=\varphi_{2}^{2}(\delta_{3})=1, \varphi_{1}^{3}(\delta_{4})=\ldots=\varphi_{4}^{3}(\delta_{7})=1$, and so on, and consider the sequences $(\psi_{n}^{k})_{n}$ defined as
\[
(\psi_{n}^{1})= (\varphi_{1}^{1},0,0,\ldots),
\]
\[
(\psi_{n}^{2})= (0,\varphi_{1}^{2},\varphi_{2}^{2},0,\ldots),
\]
\[
(\psi_{n}^{3})= (0,0,0,\varphi_{1}^{3},\varphi_{2}^{3},\varphi_{3}^{3},\varphi_{4}^{3},0,\ldots),
\]
and so on. In general,
\[
(\psi_{n}^{k})_{n}= (0,\ldots,0,\varphi_{1}^{k},\varphi_{2}^{k},\ldots,\varphi_{2^{k-1}}^{k},0,\ldots),
\]
where $\varphi_{1}^{k}$ is in the $(2^{k-1})$-th position. It is clear that $\|(\psi_{n}^{k})_{n}\|_{1}\leq1$ and therefore $\|(\psi_{n}^{k})_{n}\|_{p'}^{w}\leq1$ for all $k\in\mathds{N}$. Using (\ref{eq:1}) for $(\psi_{n}^{k})_{n}$, we obtain
\[
\sum_{n=2^{k-1}}^{2^{k}-1}|y(n)|^{p'}\leq c^{p'}, \mbox{ for all $k\in\mathds{N}$ and $y=(y(n))\in B_{\ell_{2}}$.}
\]
This is, however, impossible. Indeed, take
\[
y = (0, \ldots, 0, \frac{1}{\sqrt{2^{n}}}, \ldots, \frac{1}{\sqrt{2^{n}}},0, \ldots, 0),
\]
where $y$ is null except for the positions from $2^{k-1}$ to $(2^{k}-1)$. Obviously, $y\in B_{\ell_{2}}$ and
\[
\sum_{n=2^{k-1}}^{2^{k}-1}|y(n)|^{p'} = 2^{k-1}\Big(\frac{1}{\sqrt{2^{k-1}}}\Big)^{p'} = 2^{(k-1)(1-\frac{p'}{2})},
\]
which is certainly (much) bigger than some $c^{p'}$ for sufficiently large $k$.
\end{proof}
\end{exam}

In Corollary \ref{Uexist}, we proved that, for all Banach spaces $X$, $Y$, and $Z$, every $p$ with $1\leq p\leq\infty$, and every operator $S\in\mathcal{P}_{p'}(Z,\mathcal{L}(X,Y))$, there exists an operator $U\in\mathcal{L}(Z\hat{\otimes}_{d_{p}}X,Y)$ such that $U^{\#}=S$. The next example proves that the converse of this result, in general, does not hold. We use notation from Section \ref{s5} for $\Omega = [0,1]$: $\Sigma$ is the $\sigma$-algebra of Borel subsets of $[0,1]$ and $\mathcal{B}(\Sigma)$ is the Banach space of all bounded $\Sigma$-measurable functions on $[0,1]$.

\begin{exam} \label{ex:1}
There exists an operator $S\in\mathcal{L}(c_0, \mathcal{L}(\mathcal{B}(\Sigma), \mathcal{B}(\Sigma)))$ which is not absolutely $r$-summing for $1\leq r < \infty$ and such that, whenever $1\leq p\leq \infty$, there exists an operator $U\in\mathcal{L}(c_0\hat{\otimes}_{d_p}\mathcal{B}(\Sigma), \mathcal{B}(\Sigma))$ satisfying $U^{\#} = S$.

\begin{proof}
Denote by $(e_{n})$ the unit vector basis in $c_{0}$. Consider a linear operator $S:c_{0}\rightarrow \mathcal{L}(\mathcal{B}(\Sigma), \mathcal{B}(\Sigma))$ defined by $Se_{n} = T_{n}$ for all $n\in\mathds{N}$, where $T_{n}\in\mathcal{L}(\mathcal{B}(\Sigma), \mathcal{B}(\Sigma))$ is defined by $T_{n}\varphi = \varphi \chi_{(1/(n+1),1/n]}$, $\varphi\in \mathcal{B}(\Sigma)$.

Let $(\alpha_{n})\in c_{0}$. Then
\[
\|S(\alpha_{n})\| = \Big\| \sum_{n=1}^{\infty} \alpha_{n}T_{n} \Big\| = \sup_{\|\varphi\|\leq1} \Big\| \sum_{n=1}^{\infty} \alpha_{n}  \varphi  \chi_{(1/(n+1),1/n]}\Big\|.
\]
Let $\varphi\in\mathcal{B}(\Sigma)$ with $\|\varphi\|\leq1$. For every $t\in (0,1]$, there exists a unique $n_{0}\in\mathds{N}$ such that $t\in(1/(n_{0}+1),1/n_{0}]$. Then
\[
\Big| \sum_{n=1}^{\infty} \alpha_{n} \varphi(t) \chi_{(1/(n+1),1/n]}(t) \Big| = |\alpha_{n_{0}}\varphi(t)| \leq \|(\alpha_{n})\|\,\|\varphi\| \leq \|(\alpha_{n})\|
\]
(this inequality is trivially true for $t=0$). Thus $S$ is bounded.

In order to show that there exists an operator $U\in\mathcal{L}(c_{0}\hat{\otimes}_{d_{p}}\mathcal{B}(\Sigma),\mathcal{B}(\Sigma))$ such that $U^{\#} = S$, we check that $S$ satisfies condition (b) of Theorem \ref{t:3equiv}. Let $(\tilde{\alpha}_{n})_{n=1}^{N}\subset c_{0}$, with $\tilde{\alpha}_{n} = (\alpha_{n}(m))_{m}$, and $(\varphi_{n})_{n=1}^{N}\subset\mathcal{B}(\Sigma)$ be finite systems. Then (see, e.g., \cite[p. 35]{DJT} for the formula of $\|(\cdot)\|_{p'}^{w}$ below)
\[
\|((S\tilde{\alpha}_{n})\varphi_{n})\|_{p'}^{w} = \sup \Big\{ \Big\|\sum_{n=1}^{N} \gamma_{n} (S\tilde{\alpha}_{n})\varphi_{n}\Big\| \, : \, (\gamma_{n})\in B_{\ell_{p}^{N}}  \Big\}.
\]
For $(\gamma_{n})\in B_{\ell_{p}^{N}}$, notice that
\[
\Big\|\sum_{n=1}^{N} \gamma_{n} (S\tilde{\alpha}_{n})\varphi_{n}\Big\| = \Big\|\sum_{n=1}^{N} \gamma_{n} \sum_{m=1}^{\infty} \alpha_{n}(m) T_{m}\varphi_{n}\Big\|=
\]
\[
= \Big\|\sum_{n=1}^{N} \gamma_{n} \sum_{m=1}^{\infty} \alpha_{n}(m) \varphi_{n} \chi_{(1/(m+1),1/m]}\Big\|.
\]
But, again, for every $t\in (0,1]$, there exists a unique $m_{0}\in\mathds{N}$ such that $t\in(1/(m_{0}+1),1/m_{0}]$ (for $t=0$, the chain of inequalities below is trivially true), and we obtain
\[
\Big|\sum_{n=1}^{N} \gamma_{n} \sum_{m=1}^{\infty} \alpha_{n}(m) \varphi_{n}(t) \chi_{(1/(m+1),1/m]}(t)\Big| = \Big|\sum_{n=1}^{N} \gamma_{n}  \alpha_{n}(m_{0}) \varphi_{n}(t)\Big|
\]
\[
\leq \|(\varphi_{n})\|_{\infty}\Big|\sum_{n=1}^{N} \gamma_{n}  \alpha_{n}(m_{0}) \Big| \leq \|(\varphi_{n})\|_{\infty}\Big\|\sum_{n=1}^{N} \gamma_{n}  \tilde{\alpha}_{n} \Big\| \leq \|(\varphi_{n})\|_{\infty} \|(\tilde{\alpha}_{n})\|_{p'}^{w},
\]
yielding
\[
\|((S\tilde{\alpha}_{n})\varphi_{n})\|_{p'}^{w} \leq \|(\varphi_{n})\|_{\infty} \|(\tilde{\alpha}_{n})\|_{p'}^{w},
\]
meaning that condition (b) of Theorem \ref{t:3equiv} holds.

Finally, let $1\leq r <\infty$. Since $(e_{n})\in\ell_{1}^{w}(c_{0})$, also $(e_{n})\in\ell_{r}^{w}(c_{0})$. If now $S$ were an absolutely $r$-summing operator, then the sequence $(Se_{n}) = (T_{n})$ would be absolutely $r$-summing (see, e.g., \cite[Proposition 2.1, p. 34]{DJT}. This is however impossible, because $\|T_{n}\| = 1$ for all $n\in\mathds{N}$. Thus $S$ is not absolutely $r$-summing.
\end{proof}
\end{exam}

In Proposition \ref{p:condnecess}, for a given operator $S\in\mathcal{L}(\mathcal{C}(\Omega),\mathcal{L}(X,Y))$, we obtained a necessary condition, expressed in terms of the representing measure of $S$, for the existence of an operator $U\in\mathcal{L}(\mathcal{C}_{p}(\Omega,X),Y)$ such that $U^{\#}=S$. The next example proves that this condition, in general, is not sufficient.

\begin{exam} \label{ex:2}
Let $1< p<2$.  There exist a compact Hausdorff space $\Omega$ and an operator $S\in\mathcal{L}(\mathcal{C}(\Omega), \mathcal{L}(\ell_{p},\mathds{K}))$ such that its representing measure $m:\Sigma\rightarrow \mathcal{L}(\ell_{p},\mathds{K})$ verifies that
\[
\|(m(E_{i})x_{i})\|_{p'}\leq c, \mbox{ hence also } \|(m(E_{i})x_{i})\|_{p'}^{w}\leq c,
\]
for some constant $c>0$ and for every sequence $(E_{i})$ of pairwise disjoint sets in $\Sigma$ and every sequence $(x_{i})$ in  $B_{\ell_{p}}$. However, there does not exist any operator $U\in\mathcal{L}(\mathcal{C}_{p}(\Omega,\ell_{p}),\mathds{K})$ such that $U^{\#}=S$.

\begin{proof}
Kwapie\'{n} showed (see \cite[Theorem 7]{K} or, e.g., \cite[p. 208]{DJT}) that there exist bounded linear operators $S$ from $\ell_{\infty}$ to $\ell_{p'} = \mathcal{L}(\ell_{p},\mathds{K})$, $2<p'<\infty$, which are not absolutely $p'$-summing. Since $\ell_{\infty}$ can be identified with $\mathcal{C}(\Omega)$, where $\Omega$ is the Stone--\v{C}ech compactification of $\mathds{N}$, we have an operator $S\in\mathcal{L}(\mathcal{C}(\Omega),\ell_{p'})\setminus \mathcal{P}_{p'}(\mathcal{C}(\Omega),\ell_{p'})$.

The canonical identification $\mathcal{C}_{p}(\Omega,\ell_{p})^{*} = \mathcal{P}_{p'}(\mathcal{C}(\Omega),\ell_{p}^{*})$ identifies $U\in \mathcal{C}_{p}(\Omega,\ell_{p})^{*}$ with $U^{\#}\in\mathcal{P}_{p'}(\mathcal{C}(\Omega),\ell_{p'})$ (see Section \ref{s3}). Since our $S$ is not absolutely $p'$-summing, there does not exist $U\in\mathcal{C}_{p}(\Omega,\ell_{p})^{*} = \mathcal{L}(\mathcal{C}_{p}(\Omega,\ell_{p}),\mathds{K})$ such that $U^{\#} = S$.

On the other hand, thanks to Lindenstrauss and Pe{\l}czy\'{n}ski \cite{LP} and Maurey \cite{Ma} (see, e.g., \cite[Theorems 10.6 and 10.9]{DJT}), we know that $\mathcal{L}(\mathcal{C}(\Omega), \ell_{p'}) = \mathcal{P}_{(p',1)}(\mathcal{C}(\Omega), \ell_{p'})$ (here we use again that $2<p'<\infty$). Using that $S$ is an absolutely $(p',1)$-summing operator, we shall show that its representing measure $m:\Sigma\rightarrow \ell_{p'}$ verifies the above condition. Let $(E_{i})$ be a sequence of pairwise disjoint sets in $\Sigma$ and let $(x_{i})\subset B_{\ell_{p}}$. Then, the sequence $(\chi_{E_{i}}) \in \ell_{1}^{w}(\mathcal{B}(\Sigma))$ and $\|(\chi_{E_{i}})\|_{1}^{w}\leq 1$ (see the proof of Proposition \ref{p:condnecess} or Remark \ref{r:4.3}). As $m:\Sigma\rightarrow \ell_{p'} = (\ell_{p'})^{**}$ is the (classical) representing measure of $S$, the integration operator $\hat{S} \in \mathcal{L}(\mathcal{B}(\Sigma),\ell_{p'})$ coincides with $S^{**}|_{\mathcal{B}(\Sigma)}$. Since $S$ is absolutely $(p',1)$-summing, also $S^{**}$ is (see \cite[Theorem 3.4]{A} or, e.g., \cite[17.1.5]{P}), and hence $\hat{S}$ is, giving that
\[
\|(\hat{S}\chi_{E_{i}})\|_{p'} \leq \|\hat{S}\|_{\mathcal{P}_{(p',1)}}\|(\chi_{E_{i}})\|_{1}^{w} \leq \|\hat{S}\|_{\mathcal{P}_{(p',1)}} =: c.
\]
Therefore
\[
\|(m(E_{i})x_{i})\|_{p'} = \|((\hat{S}\chi_{E_{i}})x_{i})\|_{p'} \leq \|(\hat{S}\chi_{E_{i}})\|_{p'} \leq c.
\]
\end{proof}
\end{exam}

\section*{Acknowledgements}

Fernando Mu\~{n}oz wishes to acknowledge a partial support from institutional research funding IUT20-57 of the Estonian Ministry of Education and Research and the warm hospitality provided by Eve Oja and the members of her Functional Analysis team at the University of Tartu, where a part of this work was done in the spring of 2015. The research of Eve Oja was partially supported by Estonian Science Foundation Grant 8976 and by institutional research funding IUT20-57 of the Estonian Ministry of Education and Research. The research of C\'{a}ndido Pi\~{n}eiro and Fernando Mu\~{n}oz was partially supported by the Junta de Andaluc\'{\i}a P.A.I. FQM-276.



\end{document}